\documentclass[11pt]{article}
\usepackage[tbtags]{amsmath}
\usepackage{amssymb}
\usepackage{amsthm}
\usepackage[misc]{ifsym}
\usepackage{cases}
\usepackage{mathrsfs}
\usepackage{graphicx}

\numberwithin{equation}{section}
\normalsize

\title{\bf Direct Approach of Indefinite Linear-Quadratic Mean Field Games \thanks{This work is supported by National Key R\&D Program of China (2022YFA1006104), National Natural Science Foundations of China (12271304), and Shandong Provincial Natural Science Foundations (ZR2022JQ01, ZR2020ZD24, ZR2019ZD42).}}

\author{\normalsize  Wenyu Cong\thanks{\it School of Mathematics, Shandong University, Jinan 250100, P.R. China, E-mail: congwenyu@mail.sdu.edu.cn} , Jingtao Shi\thanks{\it Corresponding author. School of Mathematics, Shandong University, Jinan 250100, P.R. China, E-mail: shijingtao@sdu.edu.cn}}

\newtheorem{mypro}{Proposition}[section]
\newtheorem{mythm}{Theorem}[section]

\begin{document}

    \maketitle

    \noindent{\bf Abstract:}\quad This paper is concerned with an indefinite linear-quadratic mean field games of stochastic large-population system, where the individual diffusion coefficients can depend on both the state and the control of the agents. Moreover, the control weights in the cost functionals could be indefinite. A direct approach is used to derive the $\epsilon$-Nash equilibrium strategy. First, we formally solving an $N$-player game problem within a vast and finite population setting. Subsequently, decoupling or reducing high-dimensional systems by introducing two Riccati equations explicitly yields centralized strategies, contingent on the state of a specific player and the average state of the population. As the population size $N$ goes infinity, the construction of decentralized strategies becomes feasible. Then, we demonstrated they are an $\epsilon$-Nash equilibrium. Numerical examples are provided to demonstrate the effectiveness of the proposed strategies.

    \vspace{2mm}


    \vspace{2mm}


    \section{Introduction}

    {\it Mean field games} (MFGs) have been attracting growing scholarly attention, finding applications in diverse fields such as system control, applied mathematics, and economics (\cite{Bensoussan-Frehse-Yam-13}, \cite{Gomes-Saude-14}, \cite{Caines-Huang-Malhame-17}, \cite{Carmona-Delarue-18}).
    MFG theory serves as a framework for describing the behavior of models characterized by large populations, where the influence of the overall population is significant, despite the negligible impact on individual entities.
    It investigates the existence of Nash equilibria.
    This is accomplished by leveraging the relationship between the finite and corresponding infinite-limit population problems.
    The methodological foundations of MFG, initially proposed by Lasry and Lions \cite{Lasry-Lions-07} and independently by Huang et al. \cite{Huang-Malhame-Caines-06}, have proven effective and tractable for analyzing weakly coupled stochastic controlled systems with mean field interactions, establishing approximate Nash equilibria.
    In particular, within the {\it linear-quadratic} (LQ) framework, MFGs provide a versatile modeling tool applicable to a wide range of practical problems.
    The solutions derived from LQ-MFGs exhibit notable and elegant properties.
    Current scholarly discourse has extensively explored MFGs, particularly within the LQ framework (\cite{Li-Zhang-08}, \cite{Bensoussan-Sung-Yam-Yung-16}, \cite{Moon-Basar-17}, \cite{Huang-Zhou-20}).
    Huang et al. \cite{Huang-Caines-Malhame-07} conducted a study on $\epsilon$-Nash equilibrium strategies in the context of LQ-MFGs with discounted costs.
    This investigation was rooted in the fixed-point approach. Subsequently, it was applied to scenarios involving long run average costs, in Li and Zhang \cite{Li-Zhang-08}.
    In the domain of MFGs featuring major players, Huang \cite{Huang-10} delved into continuous-time LQ games, providing insights into $\epsilon$-Nash equilibrium strategies.
    Huang et al. \cite{Huang-Wang-Wu-16} introduced a backward-major and forward-minor setup for an LQ-MFG, and decentralized $\epsilon$-Nash equilibrium strategies for major and minor agents were obtained.
    Huang et al. \cite{Huang-Wang-Wu-16b} examined the backward LQ-MFG of weakly coupled stochastic large population systems under both full and partial information scenarios.
    Huang and Li \cite{Huang-Li-18} delved into an LQ-MFG concerning a class of stochastic delayed systems.
    Xu and Zhang \cite{Xu-Zhang-20} explored a general LQ-MFG for stochastic large population systems, where the individual diffusion coefficient is contingent on the state and control of the agent.
    Bensoussan et al. \cite{Bensoussan-Feng-Huang-21} considered an LQ-MFG with partial observation and common noise.

    Conventionally, two approaches are employed in the resolution of mean field games.
    One is termed the fixed-point approach (or top-down approach, see \cite{Huang-Caines-Malhame-07}, \cite{Huang-Malhame-Caines-06}, \cite{Li-Zhang-08}, \cite{Bensoussan-Frehse-Yam-13}, \cite{Carmona-Delarue-18}), which initiates the process by employing mean field approximation and formulating a fixed-point equation.
    By tackling the fixed-point equation and scrutinizing the optimal response of a representative player, decentralized strategies can be formulated.
    The alternative approach is known as the direct approach (or bottom-up approach, refer to \cite{Lasry-Lions-07}, \cite{Wang-Zhang-Zhang-20}, \cite{Huang-Zhou-20}, \cite{Wang-24}).
    This method commences by formally solving an $N$-player game problem within a vast and finite population setting.
    Subsequently, by decoupling or reducing high-dimensional systems, centralized control can be explicitly derived, contingent on the state of a specific player and the average state of the population.
    As the population size $N$ approaches infinity, the construction of decentralized strategies becomes feasible.
    In the direct approach, there are no fixed-point conditions.
    In \cite{Huang-Zhou-20}, the authors discussed the connection between these two routes within an LQ setting.

    In this paper, we explore an indefinite LQ-MFG, where the individual diffusion coefficients can depend on both the state and the control of the agents, the control weights in the cost functionals could be indefinite.
    We solve the problem by directly decoupling high-dimensional {\it forward-backward stochastic differential equations} (FBSDEs) via introducing two Riccati equations, which in turn gives the centralized strategy.
    Then, inspired by the form of above Riccati equations, the limiting form of the Riccati equations is devised and its solvability is analyzed.
    The solvability of the original Riccati equations follows from the continuous dependence of the solution on the parameter.
    By the law of large numbers, we get the decentralized strategy, and verify that it is an $\epsilon$-Nash equilibrium strategy.

    The main contributions of the paper are outlined as follows.

    $\bullet$ In the drift and diffusion terms of the individual's dynamical system, both incorporate the individual's own state and control.
    There is also a coupling of the individual state to the population state average in the cost function.
    These new type bring mathematical difficulties, and, to some degree, they have some potential applications in reality.

    $\bullet$ Instead of employing the traditional fixed-point methodology such as \cite{Xu-Zhang-20}, we adopt a direct approach to address the complexities inherent in our game problem, which was used in \cite{Huang-Zhou-20}, \cite{Wang-Zhang-Zhang-20}.
    This strategic shift involves the meticulous decoupling of the high-dimensional Hamiltonian system using mean field approximations.
    By disentangling the high-dimensional Hamiltonian system through mean field approximations, we formulate a set of decentralized strategies for all players. This set is subsequently demonstrated to constitute an $\epsilon$-Nash equilibrium.

    The paper is organized as follows.
    In Section 2, we formulate the problem.
    In Section 3, we construct decentralized strategies.
    In Section 4, the proof of the asymptotic optimality is rigorously presented.
    Numerical examples are provided in Section 5 to demonstrate the effectiveness of the proposed strategies.
    Finally, some conclusions are given in Section 6.

    The following notations will be used throughout this paper.
    We use $||\cdot||$ to denote the norm of a Euclidean space, or the Frobenius norm for matrices. For a symmetric matrix $Q$ and a vector $z$, $||z||^2_Q \equiv z^\top Qz$.
    For any real-valued scalar functions $f(\cdot)$ and $g(\cdot)$ defined on $\mathbb{R}$, $f(x) = O(g(x))$ means that there exists a constant $C > 0$ such that $\lim_{x \to \infty} |\frac{f(x)}{g(x)}| = C$, where $|\cdot|$ is an absolute value, which is also equivalent to saying that there exist $C > 0$ and $x$ such that $|f(x)| \le C|g(x)|$ for any $x \ge x'$.

    Let $T>0$ be a finite time duration and $(\Omega, \mathcal{F}, \{\mathcal{H}_t\}_{0 \le t \le T}, \mathbb{P})$ be a complete filtered probability space with the filtration $\{\mathcal{H}_t\}_{0 \le t \le T}$ augmented by all the $\mathbb{P}$-null sets in $\mathcal{F}$. $\mathbb{E}$ denoted the expectation with respect to $\mathbb{P}$. Let $L_{\mathcal{H}}^2(0,T;\cdot)$ be the set of all vector-valued (or matrix-valued) $\mathcal{H}_t$-adapted processes $f(\cdot)$ such that $\mathbb{E}\int^T_0 ||f(t)||^2dt < \infty$ and $L^2_{\mathcal{H}_t}(\Omega;\cdot)$ be the set of $\mathcal{H}_t$-measurable random variables, for $t\in[0,T]$.

    \section{Problem formulation}

    We consider a large-population system with $N$ agents, where $N$ can be arbitrarily large. The states equation of the $i$th agent $\mathcal{A}_i$, is given by the following SDE:
    \begin{equation}\label{state}
        \left\{
        \begin{aligned}
            dx_i(t) =& \big[A(t)x_i(t) + B(t)u_i(t) + f(t)\big]dt\\
                     & + \big[C(t)x_i(t) + D(t)u_i(t) + g(t)\big]dW_i(t),\quad t\in[0,T],\\
            x_i(0) =&\ \xi_i ,
        \end{aligned}
        \right.
    \end{equation}
    where $1 \le i \le N$, $x_i(\cdot)$, $u_i(\cdot)$ and $\xi_i \in L^2_{\mathcal{F}_0^i}(\Omega; \mathbb{R})$ are the state process, control process and initial value (random variable) of $\mathcal{A}_i$.
    In this paper, for simplicity, we assume the dimensions of state process and control process are both one-dimensional.
    Here, $A$, $B$, $C$, $D$, $f$ and $g$ are deterministic $\mathbb{R}$-valued functions.
    $W_i(\cdot), i = 1,\cdots,N$ are a sequence of one-dimensional Brownian motions defined on $(\Omega, \mathcal{F}, \{\mathcal{F}_t\}_{0 \le t \le T}, \mathbb{P})$.
    Let $\mathcal{F}_t$ be the $\sigma$-algebra generated by $\{\xi_i, W_i(s), s \le t, 1 \le i \le N\}$.
    Denote $\mathcal{F}_t^i$ be the $\sigma$-algebra generated by $\{\xi_i, W_i(s), s \le t\}, 1 \le i \le N$.
    We define the decentralized control set for all agents and the $i$th agent $\mathcal{A}_i$ as follows:
    \begin{equation*}
    \begin{aligned}
    \mathscr{U}_d[0,T] &:= \big\{(u_1(\cdot),\cdots,u_N(\cdot))|u_i(\cdot) \in L^2_{\mathcal{F}^i}(0,T;\mathbb{R})\big\}, \\
    \mathscr{U}_d^i[0,T] &:= \big\{u_i(\cdot)|u_i(\cdot) \in L^2_{\mathcal{F}^i}(0,T;\mathbb{R})\big\},
    \end{aligned}
    \end{equation*}
    and the centralized control set for all agents and the $i$th agent $\mathcal{A}_i$:
    \begin{equation*}
    \begin{aligned}
    \mathscr{U}_c[0,T] &:= \big\{(u_1(\cdot),\cdots,u_N(\cdot))|u_i(\cdot) \in L^2_{\mathcal{F}}(0,T;\mathbb{R}), 1 \le i \le N\big\},\\
    \mathscr{U}_c^i[0,T] &:= \big\{u_i(\cdot)|u_i(\cdot) \in L^2_{\mathcal{F}}(0,T;\mathbb{R})\big\}.
    \end{aligned}
    \end{equation*}
    The cost functional of $\mathcal{A}_i$ is given by
    \begin{equation}\label{cost}
    \begin{aligned}
        J^N_i(u_i(\cdot), u_{-i}(\cdot)) &= \frac{1}{2} \mathbb{E} \left\{\int_0^T \left[\Big|\Big|x_i(t) - \Gamma(t) x^{(N)}(t) - \eta(t)\Big|\Big|^2_{Q(t)} + ||u_i(t)||^2_{R(t)} \right]dt\right. \\
            &\qquad\qquad \left.+ \Big|\Big|x_i(T) - \Gamma_0 x^{(N)}(T) - \eta_0\Big|\Big|^2_H\right\},
    \end{aligned}
    \end{equation}
    where $Q$, $R$, $\Gamma$ and $\eta$ are deterministic $\mathbb{R}$-valued functions, $H$, $\Gamma_0$ and $\eta_0$ are scalar constants.
    $x^{(N)}(t) := \frac{1}{N} \sum_{j=1}^{N} x_j(t)$ is called the state average or mean field term of all agents, $u_{-i}(\cdot) := (u_1(\cdot),\cdots,u_{i-1}(\cdot),u_{i+1}(\cdot),\cdots,u_N(\cdot))$.

    We propose the following assumptions:

    \noindent {\bf (A1)} $\{\xi_i, 1 \le i \le N \}$ is a sequence of independent and identically distributed random variables with $\mathbb{E}[\xi_i] = \bar{\xi}, i = 1,2,\cdots,N$, and there exists a constant $C$ such that $\sup_{1 \le i \le N} \mathbb{E} [||\xi_i||^2] \le C.$

    \noindent {\bf (A2)} $\{W_i(t), 1 \le i \le N \}$ are mutually independent, which are independent of $\{\xi_i, 1 \le i \le N \}$.

    \noindent {\bf (A3)} $Q(t) \ge 0,\forall\ 0 \le t \le T$, $H \ge 0$.

    In this paper, we explore the following two problems.

    \noindent {\bf (P)}: Find a Nash equilibrium strategy $u^{N*}(\cdot) := (u_1^*(\cdot),\cdots,u_N^*(\cdot))$ to (\ref{cost}), subject to (\ref{state}), i.e., for all $i$, $1 \le i \le N$,
    $$J^N_i(u_i^*(\cdot), u_{-i}^*(\cdot)) \le \inf_{u_i(\cdot) \in \mathscr{U}_c^i[0,T]} J^N_i(u_i(\cdot), u_{-i}^*(\cdot)).$$

    \noindent {\bf (P')}: Find an $\epsilon$-Nash equilibrium strategy $u^{N*}(\cdot)$ to (\ref{cost}), subject to (\ref{state}), i.e., there exists a constant $\epsilon \ge 0$ such that for all $i$, $1 \le i \le N$,
    $$J^N_i(u_i^*(\cdot), u_{-i}^*(\cdot)) \le \inf_{u_i(\cdot) \in \mathscr{U}_d^i[0,T]} J^N_i(u_i(\cdot), u_{-i}^*(\cdot)) + \epsilon.$$

    \section{Design of decentralized strategies}

    Subsequently, the time arguments of functions may be omitted if their exclusion does not lead to confusion.
    We initially derive the following result.

    \begin{mythm}\label{result1}
        Under Assumption (A1)-(A3), for the initial value $\xi_i, i = 1,\cdots,N$, {\bf (P)} admits an optimal control $\check{u}_i(\cdot) \in \mathscr{U}_c^i[0,T], i = 1,\cdots,N$, if and only if the following two conditions hold:

        (i) For $i = 1,\cdots,N$, the adapted solution $(\check{x}_i(\cdot), \check{p}_i(\cdot), \check{q}^j_i(\cdot), j = 1,\cdots,N)$ to the FBSDE
        \begin{equation}\label{adjoint FBSDE}
            \left\{
            \begin{aligned}
                d\check{x}_i &= \big[A\check{x}_i + B\check{u}_i + f\big]dt + \big[C\check{x}_i + D\check{u}_i + g\big]dW_i,\\
                d\check{p}_i &= -\left[A\check{p}_i + C\check{q}^i_i + \left(1 - \frac{\Gamma}{N}\right)Q\left(\check{x}_i - \Gamma \check{x}^{(N)} - \eta\right)\right]dt\\
                &\quad + \sum_{j=1}^{N} \check{q}_i^j dW_j,\quad t\in[0,T],\\
                \check{x}_i(0) &= \xi_i , \quad \check{p}_i(T) = H \left(1 - \frac{\Gamma_0}{N}\right) \left(\check{x}_i(T) - \Gamma_0 \check{x}^{(N)}(T) - \eta_0\right) ,
            \end{aligned}
            \right.
        \end{equation}
        satisfies the following stationarity condition:
        \begin{equation}\label{stationarity condition}
            B\check{p}_i + D\check{q}^i_i + R\check{u}_i = 0, \quad a.e.,\ a.s..
        \end{equation}

        (ii) For $i = 1,\cdots,N$, the following convexity condition holds:
        \begin{equation}\label{convexity condition}
            \mathbb{E} \left\{\int_0^T \left[Q \left(1 - \frac{\Gamma}{N}\right)^2 \tilde{x}_i^2 + Ru_i^2\right]dt + H\left(1 - \frac{\Gamma_0}{N}\right)^2 \tilde{x}_i^2(T)\right\} \ge 0, \quad \forall u_i(\cdot) \in \mathscr{U}_c^i[0,T],
        \end{equation}
        where $\tilde{x}_i(\cdot)$ represents the solution to the following SDE:
        \begin{equation}\label{tilde x}
            \left\{
            \begin{aligned}
                d\tilde{x}_i &= \left[A\tilde{x}_i + Bu_i\right]dt + \left[C\tilde{x}_i + Du_i\right]dW_i,\quad t\in[0,T],\\
                \tilde{x}_i(0) &= 0.
            \end{aligned}
            \right.
        \end{equation}
    \end{mythm}

    \begin{proof}
        We consider the agent $\mathcal{A}_i$.
        For given $\xi_i \in L^2_{\mathcal{F}_0^i}(\Omega; \mathbb{R})$, suppose that $\check{u}_i(\cdot)$ is a candidate of the optimal strategy for $\mathcal{A}_i$, suppose that $(\check{x}_i(\cdot), \check{p}_i(\cdot), \check{q}^j_i(\cdot), j = 1,\cdots,N)$ is an adapted solution to FBSDE (\ref{adjoint FBSDE}).
        For any $u_i(\cdot) \in \mathscr{U}_c^i[0,T]$ and $\theta \in \mathbb{R}$, let $x_i^{\theta}(\cdot)$ be the solution to the following perturbed state equation:
        \begin{equation}\label{perturbed state}
        \left\{\begin{aligned}
                dx_i^{\theta} &= \big[Ax_i^{\theta} + B(\check{u}_i + \theta u_i) + f\big]dt + \big[Cx_i^{\theta} + D(\check{u}_i + \theta u_i) + g\big]dW_i,\\
                x_i^{\theta}(0) &= \xi_i.
        \end{aligned}\right.
        \end{equation}
        Then, $\tilde{x}_i^{\theta}(\cdot) := \frac{x_i^{\theta}(\cdot) - \check{x}_i(\cdot)}{\theta}$ is independent of $\theta$ and satisfies (\ref{tilde x}). Applying It\^{o}'s formula to $\check{p}_i(\cdot) \tilde{x}_i(\cdot)$, we obtain
        \begin{equation*}
        \begin{aligned}
            \mathbb{E}& \bigg[H \left(1 - \frac{\Gamma_0}{N}\right) \left(\check{x}_i(T) - \Gamma_0 \check{x}^{(N)}(T) - \eta_0\right) \tilde{x}_i(T)\bigg] \\
            = \mathbb{E}& \big[\check{p}_i(T) \tilde{x}_i(T) - \check{p}_i(0) \tilde{x}_i(0)\big] \\
            = \mathbb{E}& \int_0^T \Biggl\{ -\left[A\check{p}_i + C\check{q}^i_i + \left(1 - \frac{\Gamma}{N}\right)Q\left(\check{x}_i - \Gamma \check{x}^{(N)} - \eta\right)\right]\tilde{x}_i \\
            &\qquad + \left[A\tilde{x}_i + Bu_i\right]\check{p}_i + \left[C\tilde{x}_i + Du_i\right]\check{q}^i_i \Biggr\}dt \\
            = \mathbb{E}& \int_0^T \left\{-\left(1 - \frac{\Gamma}{N}\right)Q\left(\check{x}_i - \Gamma \check{x}^{(N)} - \eta\right)\tilde{x}_i + Bu_i\check{p}_i + Du_i\check{q}^i_i \right\}dt.
        \end{aligned}
        \end{equation*}
        Therefore,
        \begin{equation*}
        \begin{aligned}
            &J^N_i(\check{u}_i(\cdot) + \theta u_i(\cdot), u_{-i}(\cdot)) - J^N_i(\check{u}_i(\cdot), u_{-i}(\cdot))\\
            =&\ \frac{1}{2} \mathbb{E} \left\{\int_0^T \left[\Big|\Big|x_i^{\theta} - \Gamma x^{(N)\theta} - \eta\Big|\Big|^2_Q + ||\check{u}_i + \theta u_i||^2_R\right]dt
            + \Big|\Big|x_i^{\theta}(T) - \Gamma_0 x^{(N)\theta}(T) - \eta_0\Big|\Big|^2_H\right\}\\
            &- \frac{1}{2} \mathbb{E} \left\{\int_0^T \left[\Big|\Big|\check{x}_i - \Gamma \check{x}^{(N)} - \eta\Big|\Big|^2_Q + ||\check{u}_i||^2_R\right]dt
            + \Big|\Big|\check{x}_i(T) - \Gamma_0 \check{x}^{(N)}(T) - \eta_0\Big|\Big|^2_H\right\} \\
            =&\ \frac{1}{2} \mathbb{E} \Biggl\{\int_0^T \left[\bigg|\bigg|\left(1 - \frac{\Gamma}{N}\right)(\check{x}_i + \theta \tilde{x}_i) - \Gamma x^{(N-1)}_{-i} - \eta\bigg|\bigg|^2_Q + ||\check{u}_i + \theta u_i||^2_R \right]dt \\
            &\qquad + \bigg|\bigg|\left(1 - \frac{\Gamma_0}{N}\right)(\check{x}_i(T) + \theta \tilde{x}_i(T)) - \Gamma_0 x^{(N-1)}_{-i}(T) - \eta_0\bigg|\bigg|^2_H \Biggr\} \\
            &- \frac{1}{2} \mathbb{E} \Biggl\{\int_0^T \left[\bigg|\bigg|\left(1 - \frac{\Gamma}{N}\right)\check{x}_i - \Gamma x^{(N-1)}_{-i} - \eta\bigg|\bigg|^2_Q + ||\check{u}_i||^2_R \right]dt \\
            &\qquad\quad + \bigg|\bigg|\left(1 - \frac{\Gamma_0}{N}\right)\check{x}_i(T) - \Gamma_0 x^{(N-1)}_{-i}(T) - \eta_0\bigg|\bigg|^2_H\Biggr\} \\
            =&\ \frac{1}{2} \theta^2 \mathbb{E} \left\{\int_0^T \bigg[Q\left(1 - \frac{\Gamma}{N}\right)^2 \tilde{x}_i^2 + Ru_i^2\bigg]dt + H \left(1 - \frac{\Gamma_0}{N}\right)^2 \tilde{x}_i^2(T)\right\} \\
            & + \theta \mathbb{E} \Biggl\{\int_0^T \left[Q\left(1 - \frac{\Gamma}{N}\right) \tilde{x}_i \left(\check{x}_i - \Gamma \check{x}^{(N)} - \eta\right) + u_iR\check{u}_i\right]dt \\
            & \qquad + H \left(1 - \frac{\Gamma_0}{N}\right) \left(\check{x}_i(T) - \Gamma_0 \check{x}^{(N)}(T) - \eta_0\right) \tilde{x}_i(T)\Biggr\} \\
            =&\ \frac{1}{2} \theta^2 \mathbb{E} \left\{\int_0^T \bigg[Q\left(1 - \frac{\Gamma}{N}\right)^2 \tilde{x}_i^2 + Ru_i^2\bigg]dt + H \left(1 - \frac{\Gamma_0}{N}\right)^2 \tilde{x}_i^2(T)\right\} \\
            & + \theta \mathbb{E} \left\{\int_0^T u_i\left[B\check{p}_i + D\check{q}^i_i + R\check{u}_i \right]dt\right\},
        \end{aligned}
        \end{equation*}
        where only here $x^{(N)\theta}(\cdot) := \frac{1}{N} \sum_{j \ne i} x_j(\cdot) + \frac{1}{N} x_i^{\theta}(\cdot)$, $\check{x}^{(N)}(\cdot) := \frac{1}{N} \sum_{j \ne i} x_j(\cdot) + \frac{1}{N} \check{x}_i(\cdot)$ and $x^{(N-1)}_{-i}(\cdot) := \frac{1}{N} \sum_{j \ne i} x_j(\cdot)$.
        Thus, we have,
        $$J^N_i(\check{u}_i(\cdot), u_{-i}(\cdot)) \le J^N_i(\check{u}_i(\cdot) + \theta u_i(\cdot), u_{-i}(\cdot))$$
        if and only if (\ref{stationarity condition}) and (\ref{convexity condition}) hold. The proof is complete.
    \end{proof}

    The next step is to obtain a proper form for deriving the decentralized feedback representation of optimal strategy.

    We denote $\xi^{(N)} := \check{x}^{(N)}(0) = \frac{1}{N} \sum_{j=1}^{N} \xi_j$, $\check{u}^{(N)}(\cdot) := \frac{1}{N} \sum_{j=1}^{N} \check{u}_j$. It follows from (\ref{adjoint FBSDE}) that
    \begin{equation}\label{checkxN}
        \left\{
        \begin{aligned}
            d\check{x}^{(N)} =& \big[A\check{x}^{(N)} + B\check{u}^{(N)} + f\big]dt + \frac{1}{N} \sum_{j=1}^{N} \big[Cx_j + Du_j + g\big]dW_j,\quad t\in[0,T],\\
            \check{x}^{(N)}(0) =&\ \xi^{(N)}.
        \end{aligned}
        \right.
    \end{equation}
    We consider the transformation
    $$\check{p}_i(\cdot) = P_N(\cdot)\check{x}_i(\cdot) + K_N(\cdot)\check{x}^{(N)}(\cdot) + \phi_N(\cdot),\quad i = 1,\cdots,N,$$
    where $P_N(\cdot)$, $K_N(\cdot)$ and $\phi_N(\cdot)$ are differential functions with $P_N(T) = H\left(1 - \frac{\Gamma_0}{N}\right)$, $K_N(T) = -H\left(1 - \frac{\Gamma_0}{N}\right)\Gamma_0$, $\phi_N(T) = -H\left(1 - \frac{\Gamma_0}{N}\right)\eta_0$.

    Then by the first equation of (\ref{adjoint FBSDE}), (\ref{checkxN}) and It\^{o}'s formula, we get
    \begin{equation}\label{checkpi}
        \begin{aligned}
            d\check{p}_i &= \dot{P}_N \check{x}_i dt + P_N\Big\{\big[A\check{x}_i + B\check{u}_i + f\big]dt + \big[C\check{x}_i + D\check{u}_i + g\big]dW_i\Big\} + \dot{K}_N \check{x}^{(N)} dt\\
            &\quad  + K_N\Big\{\big[A\check{x}^{(N)} + B\check{u}^{(N)} + f\big]dt + \frac{1}{N} \sum_{j=1}^{N} \big[Cx_j + Du_j + g\big]dW_j\Big\} + \dot{\phi}_N.
        \end{aligned}
    \end{equation}

    Comparing the coefficients of the corresponding diffusion terms with the second equation of (\ref{adjoint FBSDE}), we have
    \begin{equation}\label{checkqii}
        \check{q}^i_i = \left(P_N + \frac{1}{N} K_N\right)\big[C\check{x}_i + D\check{u}_i + g\big],
    \end{equation}
    \begin{equation}
        \check{q}^j_i = \frac{1}{N} K_N \big[C\check{x}_j + D\check{u}_j + g\big], \quad j \ne i.
    \end{equation}

    From (\ref{stationarity condition}) and (\ref{checkqii}), we have for any $i = 1,\cdots,N$,
    \begin{equation}
        \alpha_N \check{u}_i + \beta_N \check{x}_i + \gamma_N \check{x}^{(N)} + \delta_N = 0,
    \end{equation}
    where
    \begin{equation}\label{alphaN}
        \left\{
        \begin{aligned}
            &\alpha_N := R + (P_N + \frac{1}{N} K_N)D^2,\quad \beta_N := BP_N + (P_N + \frac{1}{N} K_N)CD,\\
            &\gamma_N := BK_N,\quad \delta_N := B\phi_N + (P_N + \frac{1}{N} K_N)gD.
        \end{aligned}
        \right.
    \end{equation}

    We introduce the following assumption:

    \noindent {\bf (A4)} $\alpha_N \ne 0$.

    If (A4) holds, we obtain that the optimal control is given by
    \begin{equation}
        \check{u}_i = - \alpha_N^{-1}\left[\beta_N \check{x}_i + \gamma_N \check{x}^{(N)} + \delta_N\right].
    \end{equation}
    Furthermore, we have
    \begin{equation}
        \check{u}^{(N)} = - \alpha_N^{-1}\left[(\beta_N + \gamma_N) \check{x}^{(N)} + \delta_N\right].
    \end{equation}
    This together with (\ref{checkpi}) gives
    \begin{equation}\label{PN}
        \left\{
        \begin{aligned}
            &\dot{P}_N + 2A P_N - P_N B \alpha_N^{-1} \beta_N + C \left(P_N + \frac{1}{N} K_N\right) \left(C - D \alpha_N^{-1} \beta_N\right) + \left(1 - \frac{\Gamma}{N}\right)Q = 0,\\
            &P_N(T) = H\left(1 - \frac{\Gamma_0}{N}\right),
        \end{aligned}
        \right.
    \end{equation}
    \begin{equation}\label{KN}
        \left\{
        \begin{aligned}
            &\dot{K}_N + 2A K_N - P_N B \alpha_N^{-1} \gamma_N - K_N B \alpha_N^{-1} \left(\beta_N + \gamma_N\right) \\
            &\ - CD \left(P_N + \frac{1}{N} K_N\right) \alpha_N^{-1} \gamma_N - \left(1 - \frac{\Gamma}{N}\right)Q\Gamma = 0,\\
            &K_N(T) = -H\left(1 - \frac{\Gamma_0}{N}\right)\Gamma_0,
        \end{aligned}
        \right.
    \end{equation}
    \begin{equation}\label{phiN}
        \left\{
        \begin{aligned}
            &\dot{\phi}_N + fP_N + fK_N + A \phi_N - P_N B \alpha_N^{-1} \delta_N - K_N B \alpha_N^{-1} \delta_N\\
            &\ + C \left(P_N + \frac{1}{N} K_N\right) \left(g - D \alpha_N^{-1} \delta_N\right) - \left(1 - \frac{\Gamma}{N}\right)Q\eta = 0,\\
            &\phi_N(T) = -H\left(1 - \frac{\Gamma_0}{N}\right)\eta_0.
        \end{aligned}
        \right.
    \end{equation}

    Note that (\ref{phiN}) is a linear backward {\it ordinary differential equation} (backward ODE). If (\ref{PN})-(\ref{KN}) admit solutions, then (\ref{phiN}) has a solution.
    From the above discussion, Theorem \ref{adjoint FBSDE} and Theorem 4.1 on page 47 of Ma and Yong \cite{Ma-Yong-99}, we have the following result.

    \begin{mypro}\label{Existence of uniqueness}
        Under Assumptions (A1)-(A4), if (\ref{PN})-(\ref{KN}) respectively admits a solution, then Problem {\bf (P)} admits a unique solution.
        \begin{equation}\label{optimal strategy N}
            \check{u}_i = - \alpha_N^{-1}\left[\beta_N \check{x}_i + \gamma_N \check{x}^{(N)} + \delta_N\right],\quad i = 1,\cdots,N.
        \end{equation}
    \end{mypro}

    Inspired by the discussion above, we denote
    \begin{equation}\label{alpha}
        \left\{
        \begin{aligned}
            &\alpha := R + PD^2,\quad \beta := BP + PCD,\\
            &\gamma := BK,\quad \delta := B\phi + PgD,
        \end{aligned}
        \right.
    \end{equation}
    where $P$, $K$ and $\phi$ satisfy
    \begin{equation}\label{P}
        \left\{
        \begin{aligned}
            &\dot{P} + 2AP + CPC - \alpha^{-1} \beta^2 + Q = 0,\\
            &P(T) = H,
        \end{aligned}
        \right.
    \end{equation}
    \begin{equation}\label{K}
        \left\{
        \begin{aligned}
            &\dot{K} + 2AK - KB\alpha^{-1}(\beta + \gamma)- \beta\alpha^{-1}\gamma - Q\Gamma = 0,\\
            &K(T) = -H\Gamma_0,
        \end{aligned}
        \right.
    \end{equation}
    \begin{equation}\label{phi}
        \left\{
        \begin{aligned}
            &\dot{\phi} + fP + fK + A\phi + CPg - KB \alpha^{-1}\delta - \beta\alpha^{-1}\delta - Q\eta = 0,\\
            &\phi(T) = -H\eta_0.
        \end{aligned}
        \right.
    \end{equation}

    In the following we will discuss the solvability of the Riccati equations (\ref{P}) and (\ref{K}).
    By (\ref{alpha}), we can rewrite (\ref{P}) as
    $$\dot{P} + (2A + C^2)P + Q - (R + D^2P)^{-1}(B + CD)^2P^2 = 0,\quad P(T) = H.$$
    Two results for the solvability of this Riccati equation are given in \cite{Xu-Zhang-20}.
    When the solvability condition of the Riccati equation (\ref{P}) holds, we reorganize the Riccati equation (\ref{K}) as
    \begin{equation}\label{rewriteK}
    \left\{
        \begin{aligned}
            0 =&\ \dot{K} - Q\Gamma + 2\big[A - B(R + PD^2)^{-1}(BP + PCD)\big]K -(R + PD^2)^{-1}B^2K^2, \\
            :=&\ \dot{K} + a + bK + cK^2,\\
            K(T) =& -H\Gamma_0.
        \end{aligned}
    \right.
    \end{equation}
    According to the basic conclusion of ODE, if there exists constant $M \ge 0$, such that
    $$|a(t)| + |b(t)| + |c(t)| \le M, \quad c(t) \ne 0, \quad \forall t \in [0,T],$$
    then (\ref{rewriteK}) admits a unique solution, i.e., (\ref{K}) admits a unique solution.

    By leveraging the continuous dependence of the solution on the parameter, as discussed in \cite{Huang-Zhou-20}, and under the aforementioned conditions, we ascertain that for a sufficiently large $N$, (\ref{PN})-(\ref{phiN}) admit solutions, respectively.

    Let $N \to \infty$. By the law of large numbers, we may approximate $x^{(N)}$ in (\ref{checkxN}) with $\bar{x}$, which satisfies
    \begin{equation}\label{barx}
        d\bar{x} = \big[(A - B\alpha^{-1}(\beta + \gamma))\bar{x} - B\alpha^{-1}\delta + f\big]dt, \quad \bar{x}(0) = \bar{\xi}.
    \end{equation}
    By Proposition \ref{Existence of uniqueness}, the decentralized strategy for agent $\mathcal{A}_i, i = 1,\cdots,N$ may be taken as
    \begin{equation}\label{decentralized strategies}
        \hat{u}_i = - \alpha^{-1}[\beta\hat{x}_i + \gamma\bar{x} + \delta],
    \end{equation}
    where $\hat{x}_i$ satisfies
    \begin{equation}\label{hatx}
        \left\{
        \begin{aligned}
            d\hat{x}_i =&\ \big[(A - B\alpha^{-1}\beta)\hat{x}_i - B\alpha^{-1}\gamma\bar{x} - B\alpha^{-1}\delta + f\big]dt\\
            & + \big[(C - D\alpha^{-1}\beta)\hat{x}_i - D\alpha^{-1}\gamma\bar{x} - D\alpha^{-1}\delta + g\big]dW_i,\quad t\in[0,T],\\
            \hat{x}_i(0) =&\ \xi_i.
        \end{aligned}
        \right.
    \end{equation}

    \section{$\epsilon$-Nash equilibria}

    In this section, we aim to demonstrate that the decentralized strategies (\ref{decentralized strategies}) of agent $\mathcal{A}_i, i = 1,\cdots,N$, constitute an approximated $\epsilon$-Nash equilibrium.

    \begin{mythm}\label{eNE}
        Assume that (A1)-(A4) hold. For problem {\bf (P')}, $( \hat{u}_1(\cdot), \cdots, \hat{u}_N(\cdot))$ given in (\ref{decentralized strategies}) constitutes an $\epsilon$-Nash equilibrium, where $\epsilon = O\left(\frac{1}{\sqrt{N}}\right)$.
    \end{mythm}
    \begin{proof}
        Through (\ref{hatx}), we have
        \begin{equation}\label{hatxN}
            \left\{
            \begin{aligned}
                d\hat{x}^{(N)} =&\ \big[(A - B\alpha^{-1}\beta)\hat{x}^{(N)} - B\alpha^{-1}\gamma\bar{x} - B\alpha^{-1}\delta + f\big]dt\\
                & + \frac{1}{N} \sum_{j=1}^{N} \big[(C - D\alpha^{-1}\beta)\hat{x}_i - D\alpha^{-1}\gamma\bar{x} - D\alpha^{-1}\delta + g\big]dW_j,\\
                \hat{x}^{(N)}(0) =&\ \xi^{(N)}.
            \end{aligned}
            \right.
        \end{equation}
        By (\ref{barx}) and (\ref{hatxN}), it can be verified that
        \begin{equation}\label{estimate hatxN barx}
            \mathbb{E} \int_0^T ||\hat{x}^{(N)} - \bar{x}||^2dt = O\left(\frac{1}{N}\right).
        \end{equation}
        For $i = 1, \cdots,N$, denote $\tilde{u}_i(\cdot) := u_i(\cdot) -\hat{u}_i(\cdot)$ and $\tilde{x}_i(\cdot) := x_i(\cdot) -\hat{x}_i(\cdot)$.
        Then for fixed $u_{-i}$, $\tilde{x}_i(\cdot)$ satisfies
        \begin{equation*}
            d\tilde{x}_i = \left[A\tilde{x}_i + B\tilde{u}_i\right]dt + \left[C\tilde{x}_i + D\tilde{u}_i\right]dW_i, \quad \tilde{x}_i(0) = 0,\quad i = 1,\cdots,N.
        \end{equation*}
        Thus,
        \begin{equation*}
            \sum_{i=1}^{N} \mathbb{E} \int_0^T \left(||\tilde{x}_i(t)||^2 + ||\tilde{u}_i(t)||^2\right)dt < \infty.
        \end{equation*}
        From (\ref{cost}), we have
        \begin{equation}
            J_i^N(u_i(\cdot), \hat{u}_{-i}(\cdot)) = J_i^N(\hat{u}_i(\cdot), \hat{u}_{-i}(\cdot)) + \tilde{J}_i^N(\tilde{u}_i(\cdot), \hat{u}_{-i}(\cdot)) + \mathcal{I}_i^N,
        \end{equation}
        where
        \begin{equation*}
            \tilde{J}_i^N(\tilde{u}_i(\cdot), \hat{u}_{-i}(\cdot)) := \frac{1}{2} \mathbb{E} \left\{\int_0^T \left[\bigg|\bigg|\left(1 - \frac{\Gamma}{N}\right)\tilde{x}_i\bigg|\bigg|^2_Q + ||\tilde{u}_i||^2_R\right]dt + \bigg|\bigg|\left(1 - \frac{\Gamma_0}{N}\right)\tilde{x}_i(T)\bigg|\bigg|^2_H\right\},
        \end{equation*}
        and
        \begin{equation*}
            \begin{aligned}
                \mathcal{I}_i^N := \mathbb{E} \Biggl\{&\int_0^T \bigg[Q\left(1 - \frac{\Gamma}{N}\right)\tilde{x}_i\left(\hat{x}_i - \Gamma\hat{x}^{(N)} - \eta\right) + R\tilde{u}_i\hat{u}_i \bigg]dt \\
                &+ \left(1 - \frac{\Gamma_0}{N}\right)H\tilde{x}_i(T)\left(\hat{x}_i(T) - \Gamma_0\hat{x}^{(N)}(T) - \eta_0\right)\Biggr\}.
            \end{aligned}
        \end{equation*}
        Let $\hat{p}_i = P\hat{x}_i + K\bar{x} + \phi$. By (\ref{hatx}), (\ref{barx}) and It\^{o}'s formula, we obtain
        \begin{equation*}
            \begin{aligned}
                d\hat{p}_i =&\ \big[\dot{P} + P(A - B\alpha^{-1}\beta)\big]\hat{x}_i dt + \big[\dot{K} - PB\alpha^{-1}\gamma + K(A - B\alpha^{-1}(\beta + \gamma))\big]\bar{x} dt \\
                & + \big[\dot{\phi} - PB\alpha^{-1}\delta + Pf - KB\alpha^{-1}\delta + Kf\big]dt \\
                & + P\big[(C-D\alpha^{-1}\beta)\hat{x}_i - D\alpha^{-1}\gamma\bar{x} - D\alpha^{-1}\delta + g\big]dW_i.
            \end{aligned}
        \end{equation*}
        Applying It\^{o}'s formula to $\tilde{x}_i(\cdot)\hat{p}_i(\cdot)$, we have
        \begin{equation*}
        \begin{aligned}
                &\mathbb{E} \big[H\tilde{x}_i(T)\left(\hat{x}_i(T) - \Gamma_0\bar{x}(T) - \eta_0\right)\big] \\
                =&\ \mathbb{E} \big[\tilde{x}_i(T)\big(P(T)\hat{x}_i(T) + K(T)\bar{x}(T) + \phi(T)\big) - \tilde{x}_i(0)\big(P(0)\hat{x}_i(0) + K(0)\bar{x}(0) + \phi(0)\big)\big] \\
                =&\ \mathbb{E} \int_0^T \bigg\{\left[A\tilde{x}_i + B\tilde{u}_i\right](P\hat{x}_i + K\bar{x} + \phi) + \big[\dot{P} + P(A - B\alpha^{-1}\beta)\big]\tilde{x}_i\hat{x}_i\\
                &\qquad  + \big[\dot{K} - PB\alpha^{-1}\gamma + K(A - B\alpha^{-1}(\beta + \gamma))\big]\tilde{x}_i\bar{x} \\
                &\qquad + \big[\dot{\phi} - PB\alpha^{-1}\delta + Pf - KB\alpha^{-1}\delta + Kf\big]\tilde{x}_i \\
                &\qquad + \left[C\tilde{x}_i + D\tilde{u}_i\right]P\big[(C-D\alpha^{-1}\beta)\hat{x}_i - D\alpha^{-1}\gamma\bar{x} - D\alpha^{-1}\delta + g\big] \bigg\} \\
                =&\ \mathbb{E} \int_0^T \bigg\{ \left[\dot{P} + 2AP - PB\alpha^{-1}\beta + CP(C-D\alpha^{-1}\beta)\right]\tilde{x}_i\hat{x}_i + \big[B\hat{p}_i + D\hat{q}^i_i\big]\tilde{u}_i \\
                &\qquad + \left[\dot{K} + 2KA - PB\alpha^{-1}\gamma - KB\alpha^{-1}(\beta + \gamma) - CPD\alpha^{-1}\gamma\right]\tilde{x}_i\bar{x} \\
                &\qquad + \biggl[\dot{\phi} + A\phi + Pf + Kf - PB\alpha^{-1}\delta - KB\alpha^{-1}\delta - CPD\alpha^{-1}\delta + CPg\biggr]\tilde{x}_i \bigg\}dt \\
                =&\ \mathbb{E} \int_0^T \bigg\{-Q\tilde{x}_i\hat{x}_i + Q\Gamma\tilde{x}_i\bar{x} + Q\eta\tilde{x}_i - R\hat{u}_i\tilde{u}_i \bigg\}dt \\
                =&\ \mathbb{E} \int_0^T \bigg\{-Q\tilde{x}_i(\hat{x}_i - \Gamma\bar{x} - \eta) - R\hat{u}_i\tilde{u}_i \bigg\}dt,
        \end{aligned}
        \end{equation*}
        where $\hat{q}^i_i := P\big(C\hat{x}_i + D\hat{u}_i + g\big)$. The penultimate equation is due to (\ref{stationarity condition}) and (\ref{P})-(\ref{phi}).
        Then,
        \begin{equation*}
            \begin{aligned}
                \mathcal{I}_i^N =&\ \mathbb{E} \bigg\{\int_0^T \bigg[Q\left(1 - \frac{\Gamma}{N}\right)\tilde{x}_i\big(\hat{x}_i - \Gamma\bar{x} - \eta\big) \
                + Q\left(1 - \frac{\Gamma}{N}\right)\tilde{x}_i\Gamma\left(\bar{x} - \hat{x}^{(N)}\right) + R\tilde{u}_i\hat{u}_i \bigg]dt \\
                &\quad + H\left(1 - \frac{\Gamma_0}{N}\right)\tilde{x}_i(T)\big(\hat{x}_i(T) - \Gamma_0\bar{x}(T) - \eta_0 \big)\\
                &\quad + H\left(1 - \frac{\Gamma_0}{N}\right)\tilde{x}_i(T)\Gamma_0\left(\bar{x}(T) - \hat{x}^{(N)}(T)\right)\bigg\} \\
                =&\ \mathbb{E} \biggl\{\int_0^T \biggl[\frac{\Gamma}{N}R\hat{u}_i\tilde{u}_i + Q\left(1 - \frac{\Gamma}{N}\right)\tilde{x}_i\Gamma\left(\bar{x} - \hat{x}^{(N)}\right)\biggr]dt \\
                &\quad + \frac{\Gamma - \Gamma_0}{N} H\tilde{x}_i(T)\big(\hat{x}_i(T) - \Gamma_0\bar{x}(T) - \eta_0 \big) \\
                &\quad + H\left(1 - \frac{\Gamma_0}{N}\right)\tilde{x}_i(T)\Gamma_0\left(\bar{x}(T) - \hat{x}^{(N)}(T)\right)\biggr\} = O\left(\frac{1}{\sqrt{N}}\right).
            \end{aligned}
        \end{equation*}
        Thereby,
        $$J_i^N(\hat{u}_i(\cdot), \hat{u}_{-i}(\cdot)) \le J_i^N(u_i(\cdot), \hat{u}_{-i}(\cdot)) + \epsilon.$$
        Thus, $(\hat{u}_1(\cdot),\cdots,\hat{u}_N(\cdot))$ is an $\epsilon$-Nash equilibrium, where $\epsilon = O\left(\frac{1}{\sqrt{N}}\right)$.
    \end{proof}

    \section{Numerical examples}

    In this section, an example of numerical simulation of Problem {\bf (P')} is given to verify the conclusions we obtained in the previous section.
    Let $A=B=f=C=D=g=Q=R=H=\Gamma=\eta=\Gamma_0=\eta_0=1$.
    The time interval is set to $[0,10]$.
    Each player's initial state is drawn independently from a uniform distribution on $[0,20]$.
    The curve of $P(t)$ and $K(t)$, described by (\ref{P})-(\ref{K}), are shown in Figure \ref{PandK_plot}.
    \begin{figure}[h]
        \centering\includegraphics[width=12cm]{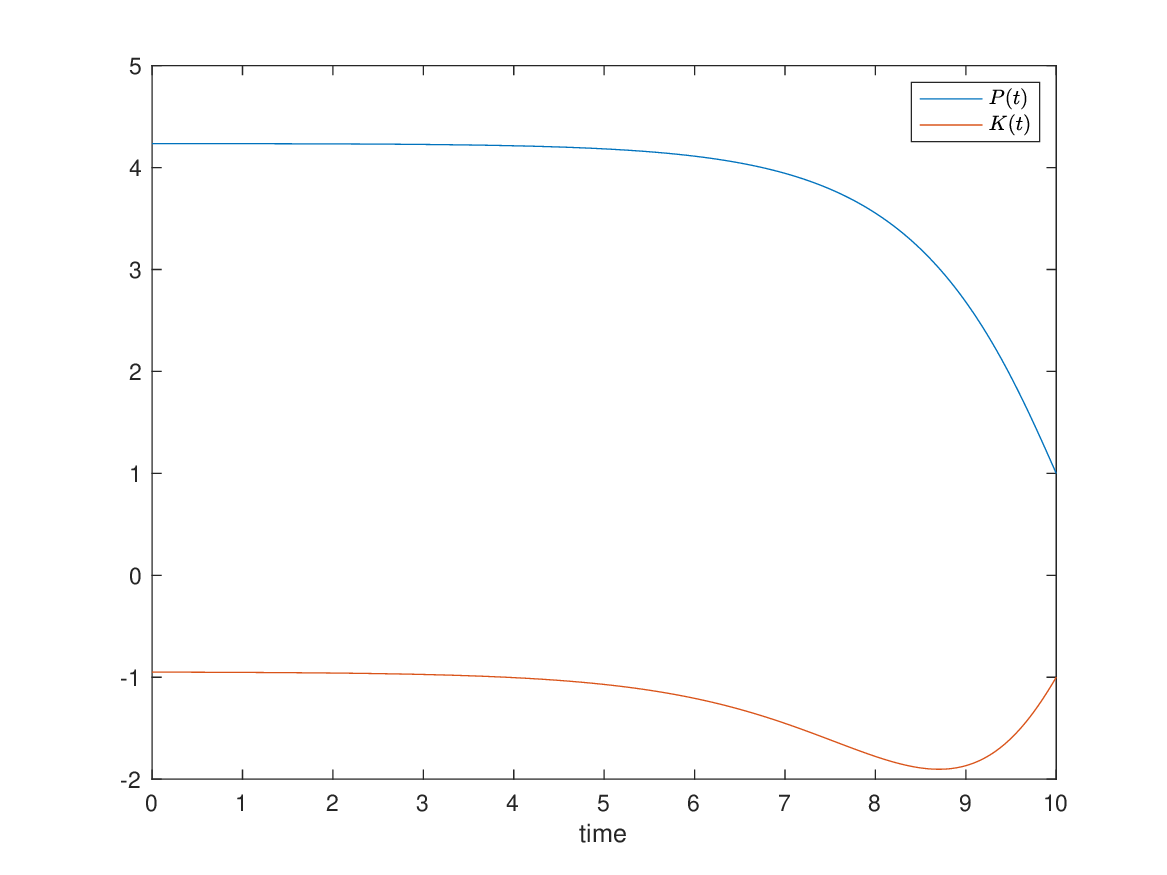}
        \caption{The curves of $P(t)$ and $K(t)$}
        \label{PandK_plot}
    \end{figure}

    We denote the performance of the decentralized strategy by $\epsilon(N) = \bigl(\mathbb{E}\int_{0}^{T}||\hat{x}^{(N)} - \bar{x}||^2 dt\bigr)^{\frac{1}{2}}$.
    A plot of $\epsilon(N)$ with respect to $N$ is shown in Figure \ref{epsilon_plot}, which confirms the consistency of the mean field approximation.

    \begin{figure}[h]
        \centering\includegraphics[width=12cm]{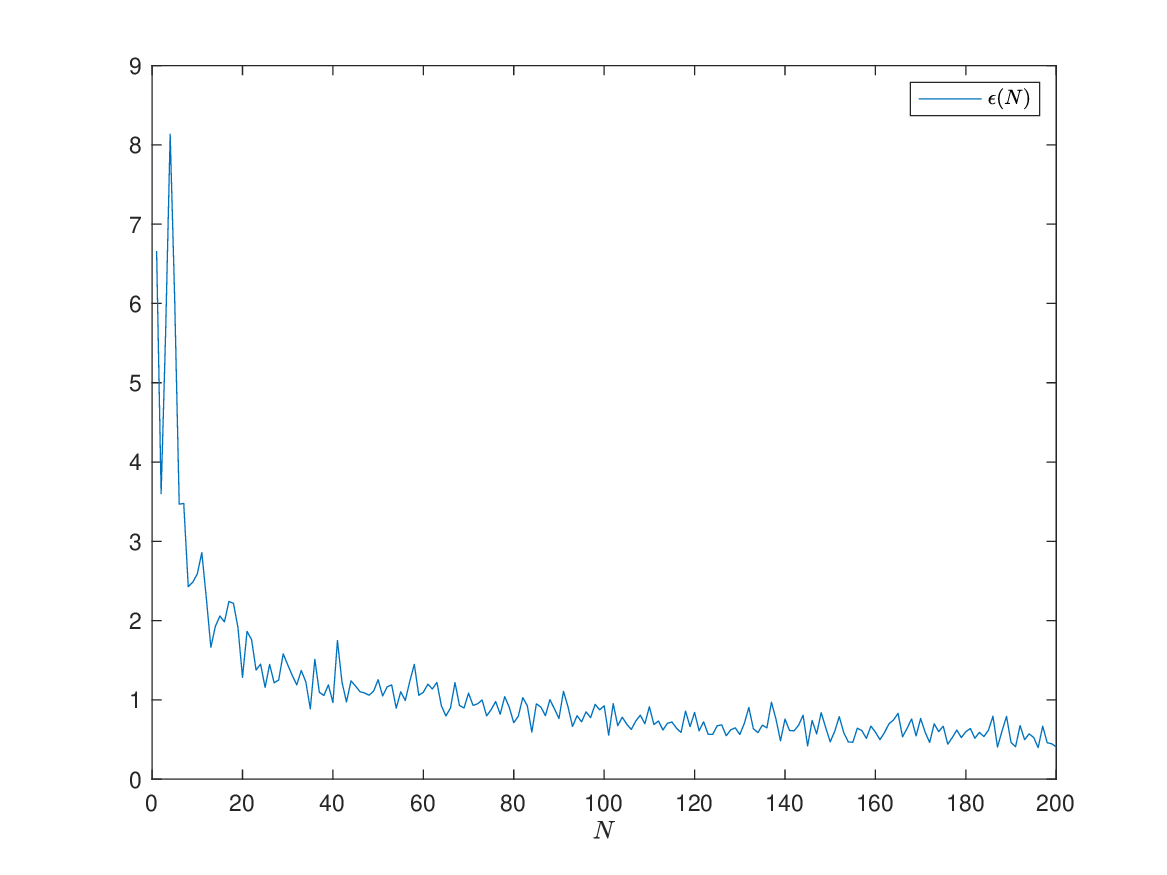}
        \caption{$\epsilon(N)$ with respect to $N$}
        \label{epsilon_plot}
    \end{figure}

    \section{Conclusions}

    In this paper, we have investigated an indefinite LQ mean field stochastic differential game.
    We use the direct approach to tackle this problem, and the resultant decentralized strategy is demonstrated to be an $\epsilon$-Nash equilibrium.
    Direct approach (\cite{Huang-Zhou-20}, \cite{Wang-Zhang-Zhang-20}, \cite{Wang-24}) applying to more kinds of mean field games is our future research interest.

\end{document}